\documentclass{amsart}

\usepackage{color}

\newtheorem{theorem}{Theorem}[section]

\theoremstyle{definition}

\newtheorem{corollary}[theorem]{Corollary}
\newtheorem{proposition}[theorem]{Proposition}

\theoremstyle{remark}
\newtheorem{remark}[theorem]{Remark}

\numberwithin{equation}{section}

\newcommand{\R}{\mathbb{R}}

\newcommand{\ria}{\rightarrow}

\newcommand{\n}{\nabla}
\newcommand{\ran}{\rangle}
\newcommand{\lan}{\langle}

\DeclareMathOperator{\di}{div}
\DeclareMathOperator{\tr}{tr}

\begin{document}

\title{Stable hypersurfaces with zero scalar curvature in Euclidean space}


\author{Hil\'ario Alencar}
\address{Instituto de Matem\' atica, Universidade Federal de Alagoas, Macei\'o, AL, 57072-900, Brasil}
\curraddr{}
\email{hilario@mat.ufal.br}
\thanks{Hil\'ario Alencar and Manfredo do Carmo were partially supported by CNPq of Brazil}

\author{Manfredo do Carmo}
\address{Instituto Nacional de Matem\' atica Pura e Aplicada, Rio de Janeiro, RJ,  22460-320, Brasil}
\curraddr{}
\email{manfredo@impa.br}

\author{Greg\'orio Silva Neto}
\address{Instituto de Matem\' atica, Universidade Federal de Alagoas, Macei\'o, AL, 57072-900, Brasil}
\curraddr{}
\email{gregorio@im.ufal.br}
\thanks{}

\subjclass[2010]{Primary 53C42, Secondary 53C40}

\date{January 23, 2015}

\dedicatory{}

\commby{}

\begin{abstract}
In this paper we prove some results concerning stability of hypersurfaces in the four dimensional Euclidean space with zero scalar curvature. First we prove there is no complete stable hypersurface with zero scalar curvature, polynomial growth of integral of the mean curvature, and with the Gauss-Kronecker curvature bounded away from zero. We conclude this paper giving a sufficient condition for a regular domain to be stable in terms of the mean and the Gauss-Kronecker curvatures of the hypersurface and the radius of the smallest extrinsic ball which contains the domain.
\end{abstract}

\maketitle

\section{Introduction}
\label{intro}
Let $M^3$ be a hypersurface of $\mathbb{R}^4$ with scalar curvature $R=0$ and whose mean curvature $H$ is nowhere zero. Let $\Omega\subset M$ be a regular domain, i.e., a domain with compact closure and piecewise smooth boundary. We recall that hypersurfaces of $\R^4$ with zero scalar curvature are critical points of the functional
$$
\mathcal{A}_1(\Omega) = \int_{\Omega} H dM
$$
under all compactly supported variations in $\Omega,$ see \cite{AdCE}. Thus, the notion of stability makes sense and we can ask for a condition to ensure that a regular domain $\Omega\subset M$ be stable. 

Since $H$ is nowhere zero, depending on choice of orientation we have $H>0$ or $H<0$ everywhere. Let $\Omega\subset M$ be a regular domain. If we choose an orientation such that $H>0$ everywhere, then the domain $\Omega$ will be stable if $\left.\dfrac{d^2\mathcal{A}_1}{dt^2}\right|_{t=0}>0$ under all compactly supported variations in $\Omega.$ Otherwise, i.e., if we choose an orientation such that $H<0,$ then the domain $\Omega$ is stable if $\left.\dfrac{d^2\mathcal{A}_1}{dt^2}\right|_{t=0}<0$ under all such variations. We say that $M$ is stable if all regular domains of $M$ are stable. For more information about the concept of stability, we refer to \cite{Smale, reilly, HL2, AdCE}. 

We say that $M^3$ has polynomial growth of the $1-$volume if there exist constants $C>0$ and $\alpha>0$ such that
\[
\int_{\mathcal{B}_r(q)}HdM\leq C r^\alpha
\]
for all $r>0$ and $q\in M,$ where $\mathcal{B}_r(q)$ denotes the geodesic ball of $M$ with center $q$ and radius $r.$ 

If $M^3$ is a hypersurface of $\R^4$ with zero scalar curvature and Gauss-Kronecker curvature nowhere zero then, by the Gauss equation $(3H)^2 = \|A\|^2 + 6R$ (where $\|A\|$ is the matrix norm of the second fundamental form and $R$ is the scalar curvature of $M^3$), the mean curvature is nowhere zero. Thus, the notion of stability makes sense in this case. 

Alencar, do Carmo and Elbert, see \cite[p.215]{AdCE}, posed a conjecture, which for the case $n=3$ can be written as follows:

\emph{
There is no complete, stable hypersurface $M^3$ of $\R^4$ with zero scalar curvature and everywhere non-zero Gauss Kronecker curvature.
}

Our first result is a partial answer to this conjecture.

\begin{theorem}\label{T-1}
There is no stable, complete, non-compact, hypersurface $M^3$ of $\R^4$ with zero scalar curvature, Gauss-Kronecker curvature bounded away from zero, and polynomial growth of the $1-$volume.
\end{theorem}

Since the image of the Gauss map for graphs lies in a open hemisphere, they are stable, see \cite[p.201, Theorem 1.1]{AdCE}. Moreover, by \cite[p.3310, Proposition 4.2]{A-S-Z}, graphs have polynomial growth of the $1-$volume. Thus an immediate corollary of Theorem \ref{T-1} is the following Bernstein type result.
\begin{corollary}
There is no complete graph in $\R^4$ with zero scalar curvature and Gauss-Kronecker curvature bounded away from zero.
\end{corollary}

\begin{remark}
{\normalfont
We point out that some condition on the Gauss-Kronecker curvature is needed. In fact, cylinders over positively curved curves are examples of stable hypersurfaces with zero scalar curvature and everywhere zero Gauss-Kronecker curvature. If we choose, for example, the curve as the graph of the polynomial function $y(x)=x^2,$ see \cite[p.492, Example 4.2]{SN}, we obtain a cylinder which is a graph with polynomial growth of the $1-$volume.  
}
\end{remark}

\begin{remark}
{\normalfont
In the direction of Theorem \ref{T-1}, by using a technique which holds only in dimension 3, the third author, see \cite[p.483, Theorem A]{SN}, proved the following result:

\emph{There is no stable complete hypersurface $M^3$ of $\R^4$ with zero scalar curvature, polynomial volume growth and such that
\[
-\frac{K}{H^3}\geq c>0
\]
everywhere, for some constant $c>0$. Here $H$ denotes the mean curvature and $K$ denotes the Gauss-Kronecker curvature of the immersion.}
}
\end{remark}

Let $r_\Omega$ be the radius of the smallest extrinsic ball which contains the domain $\Omega\subset M.$ Our second result gives a sufficient condition for a regular domain to be stable in terms of the mean and the Gauss-Kronecker curvatures of the hypersurface and the radius of the smallest extrinsic ball which contains the domain. We have the following result.

\begin{theorem}\label{stability}
Let $M^3$ be a hypersurface of $\R^4$ with zero scalar curvature and such that $H\neq0.$ Let $\Omega\subset M$ be a regular domain. If 
\[
\sup_\Omega\left(\frac{-3K}{H}\right)\leq \dfrac{3}{2r_\Omega^2},
\]
then $\Omega$ is stable.
\end{theorem}
\begin{remark}
{\normalfont
By using essentially the same proofs, the Theorems \ref{T-1} and \ref{stability} extend to the case of a hypersurface $M^n$ in $\R^{n+1},$ $n$ arbitrary, with zero scalar curvature and non-vanishing of the third symmetric function of the principal curvatures, rather than non-vanishing of Gauss-Kronecker curvature.
}
\end{remark}
%
\noindent{\it Acknowledgements.} The authors would like to thank to the referee for the comments.

\section{Proofs of the Theorems}
\label{sec:1}

In what follows we introduce a second order differential operator which will play a role for hypersurfaces with zero scalar curvature similar to that of Laplacian for minimal hypersurfaces. For that, consider the linear operator $P_1:TM\ria TM$ given by
\[
P_1 = 3HI-A,
\]
where $A:TM\ria TM$ is the linear operator associated with the second fundamental form of the immersion of $M^3$ into $\R^4$ and $I:TM\ria TM$ is the identity operator. We define
\begin{equation}\label{def.L1}
L_1(f) = \di(P_1(\n f)),
\end{equation}
where $\di X$ denotes the divergence of the vector field $X$ and $\n f$ denotes the gradient of the function $f$ in the induced metric. By a result of Hounie and Leite, see \cite[p.873, Proposition 1.5]{HL}, when the scalar curvature $R=0,$ the differential operator $L_1$ is elliptic if and only if $K\neq0$ everywhere. Since $L_1$ is an elliptic and self-adjoint operator, it has a discrete spectrum and thus we can consider the eigenvalues of $L_1$ for regular domains $\Omega\subset M.$ The first eigenvalue $\lambda_1^{L_1}(\Omega)$ has an associated positive eigenfunction $g$, i.e., a function such that $L_1g + \lambda_1^{L_1}(\Omega)g=0$ in $\Omega.$ Set 
\[
\|g\|_{H_0^1}=\left(\int_\Omega (|g|^2 + |\n g|^2) dM\right)^{1/2}
\]
and let $H_0^1(\Omega)$ be the completion of $\mathcal{C}^\infty_0(\Omega)$ with respect to the norm $\|\cdot\|_{H_0^1}.$ It is well known that $H_0^1(\Omega)$ is called a Sobolev space over $\Omega.$ It can be proven, see \cite[p.1052, Lemma 4(a)]{Smale}, that
\[
\begin{split}
\lambda_1^{L_1}(\Omega) &=\inf\left\{\frac{\int_\Omega -gL_1g dM}{\int_\Omega g^2 dM}: g\in H_0^1(\Omega) ,\ g\not\equiv0\right\}\\
& = \inf\left\{\frac{\int_\Omega \lan P_1(\n g),\n g\ran dM}{\int_\Omega g^2 dM}:g\in H_0^1(\Omega),\ g\not\equiv0 \right\}.\\
\end{split}
\]

\begin{proof}[Proof of Theorem \ref{T-1}.]
Let $M^3$ be a complete and non-compact hypersurface of $\R^4$ with zero scalar curvature and Gauss-Kronecker curvature $K\neq0$ everywhere. It is known that 
\begin{equation}\label{desig}
HK\leq \frac{3}{2}R^2.
\end{equation}
In fact, if $k_1, \ k_2,$ and $k_3$ denotes the principal curvatures of the hypersurface, then
\[
\begin{split}
(3R)^2 - 6HK& = (k_1k_2 + k_1k_3 + k_2k_3)^2 - 2(k_1 + k_2 + k_2)(k_1k_2k_3)\\
            & = [k_1^2k_2^2 + k_1^2k_3^2 + k_2^2k_3^2 + 2(k_1^2k_2k_3 + k_1k_2^2k_3 + k_1k_2k_3^2)]\\
            &\qquad - 2(k_1^2k_2k_3 + k_1k_2^2k_3 + k_1k_2k_3^2)\\
            & = k_1^2k_2^2 + k_1^2k_3^2 + k_2^2k_3^2 \geq0,
\end{split}
\]
and the equality holds if, and only if, two of the principal curvatures are zero. If, without loss of generality, we choose an orientation of $M^3$ such that $H>0,$ then the hypothesis $R=0,$ $K\neq0,$ and the (\ref{desig}) imply $K<0.$

The proof of the Theorem will be made by showing the existence of unstable domains in $M.$ Let $\{\Omega_i\}_{i=1}^\infty$ be a family of regular domains in $M$ such that $\Omega_i\subset \Omega_{i+1}$ and $\displaystyle{\bigcup_{i=1}^\infty \Omega_i=M.}$ The second variation of the functional $\displaystyle{\mathcal{A}_1(\Omega_i)=\int_{\Omega_i} H dM}$ is
\[
\begin{split}
\left.\frac{d^2\mathcal{A}_1}{dt^2}(g_i)\right|_{t=0} &=-\int_{\Omega_i}(g_iL_1g_i - 3Kg_i^2)dM\\
\end{split}
\]
where $g_i:M\ria\R$ is any piecewise smooth function defined on $\Omega_i$ with $g_i|_{\partial \Omega_i}=0,$ see \cite[p.207]{AdCE}. Let $g_i$ be the first eigenfunction of $L_1$ over $\Omega_i\subset M.$ Thus we have
\[
\left.\frac{d^2\mathcal{A}_1}{dt^2}(g_i)\right|_{t=0} = \lambda_1^{L_1}(\Omega_i)\int_{\Omega_i} g_i^2 dM + \int_{\Omega_i} 3K g_i^2 dM,
\]
i.e.,
\begin{equation}\label{eq.1}
\displaystyle{\frac{\left.\dfrac{d^2\mathcal{A}_1}{dt^2}(g_i)\right|_{t=0}}{\int_{\Omega_i} g_i^2 dM} = \lambda_1^{L_1}(\Omega_i) - \frac{\int_{\Omega_i}(-3K) g_i^2 dM}{\int_{\Omega_i} g_i^2 dM}}.
\end{equation}
Since $M^n$ has polynomial growth of the 1-volume, see \cite[p.259, Lemma 3.12]{Elbert}, gives $\lambda_1^{L_1}(M)=\inf\{\lambda_1^{L_1}(\Omega)| \Omega\subset M\}=0.$ By using that $\Omega_i\subset\Omega_{i+1},$ see \cite[p.1051, Lemma 2]{Smale}, we have $\lambda^{L_1}_1(\Omega_i)\geq \lambda_1^{L_1}(\Omega_{i+1}).$ This implies
\begin{equation}\label{eq.2}
\lim_{i\ria\infty} \lambda_1^{L_1}(\Omega_i)=0.
\end{equation}
The second member of the expression
\begin{equation}\label{eq.3}
\lim_{i\ria\infty}\left\{ \lambda_1^{L_1}(\Omega_i) - \frac{\int_{\Omega_i}(-3K) g_i^2 dM}{\int_{\Omega_i} g_i^2 dM}\right\}
\end{equation}
is the limit of the \emph{mean value} of $3K$ in $\Omega_i$ with respect to the volume element $g_i^2dM.$ There are three possibilities for the limit of the quotient of the integrals in $(\ref{eq.3}):$
\begin{itemize}
\item[(i)] It may be infinite, in which case, because $\lambda_1^{L_1}(\Omega_i)\ria0,$ the expression (\ref{eq.3}) is negative after some $i_0;$
\item[(ii)] It may be finite but non-zero, in which case, by the same reason, the expression is negative after some $i_0;$
\item[(iii)] It might be zero. Then we use for the first time the hypothesis that $K$ is bounded away from zero to conclude this case cannot happen.
\end{itemize} 

Therefore, $M$ is unstable, thus proving Theorem \ref{T-1}.
\end{proof}

In order to prove Theorem \ref{stability} we need the following Poincar\'e type inequality.

\begin{proposition}\label{mod.poincare}
Let $M^3$ be a hypersurface of $\R^4$ such that $H>0$ and $R=0.$ Let $\Omega\subset M$ be a regular domain. If $u\in H_0^1(\Omega)$ is a non-negative function, then
\begin{equation}\label{new.sobolev.2}
\int_\Omega uHdM \leq \frac{r_\Omega}{\sqrt{6}}\int_\Omega H^{1/2}\left\lan P_1(\n u),\n u\right\ran^{1/2}dM, 
\end{equation}
where $r_\Omega$ denotes the radius of the smallest extrinsic ball which contains $\Omega.$
\end{proposition}

\begin{proof}
Initially, let $B_{r_\Omega}(x_0), x_0\in\R^4$ be the smallest ball of $\R^4$ containing $\Omega$ and $\rho(x)=\rho(x_0,x)$ be the extrinsic distance from $x_0$ to $x\in M.$ Since $\Omega\subset B_{r_\Omega}(x_0),$ then, for all $x\in\Omega,$
\begin{equation}\label{diam}
\rho(x)\leq r_\Omega.
\end{equation}
We claim that
\[
\di_M(P_1(\rho\n\rho)) = 6H.
\]
In fact,
\[
\begin{split}
\di(P_1(\rho\n\rho))&=\di(3H\rho\n\rho-A(\rho\n\rho))\\
                    &=3\lan\n H,\rho\n\rho\ran + 3H\di(\rho\n\rho) - \di(A(\rho\n\rho)).
\end{split}
\]
Since $\di(A(\rho\n\rho))=3H + 3\lan \rho\n\rho,\n H\ran,$ we obtain our claim. This implies
\[
\begin{split}
\di(uP_1(\rho\n\rho))&=u\di(P_1(\rho\n\rho)) + \lan\n u,P_1(\rho\n\rho)\ran\\
                     &=6uH + \lan\n u,P_1(\rho\n\rho)\ran.\\
\end{split}
\]
Integrating the expression above over $\Omega$ and by using the divergence theorem, we have
\[
0=6\int_\Omega uH dM + \int_\Omega\lan\n u,P_1(\rho\n\rho)\ran dM,
\]
i.e.,
\begin{equation}\label{p-1}
\int_\Omega uH dM = \frac{1}{6}\int_\Omega\lan\n u,P_1(-\rho\n\rho)\ran dM.\\
\end{equation}
Since $R=0$ and $H>0,$ then $P_1$ is positive semi-definite. In fact, if $R=0$ then $(3H)^2 = |A|^2 \geq k_i^2,$  for all $i=1,2,3,$ where $k_i$ are the principal curvatures of $M^3.$ Thus $0\leq (3H)^2 - k_i^2 = (3H - k_i)(3H + k_i)$ which implies that all eigenvalues $3H - k_i$ of $P_1$ are non-negative, provided $H>0,$ i.e., $P_1$ is positive semi-definite. Thus, by using Cauchy-Schwarz inequality, we obtain
\[
\begin{split}
\lan \n u, P_1(-\rho\n\rho)\ran & = \lan\sqrt{P_1}(\n u),\sqrt{P_1}(-\rho\n\rho)\ran\\
                                & \leq |\sqrt{P_1}(\n u)||\sqrt{P_1}(-\rho\n\rho)|\\
                                & = \lan P_1(\n u),\n u\ran^{1/2}\lan P_1(\rho\n\rho),\rho\n\rho\ran^{1/2}\\
                                & \leq (\tr_M P_1)^{1/2}\rho \lan P_1(\n u),\n u\ran^{1/2} |\n \rho|\\
                                & \leq \sqrt{6}H^{1/2}\rho \lan P_1(\n u),\n u\ran^{1/2}.
\end{split}
\]
Introducing inequality above into (\ref{p-1}), we have
\[
\int_\Omega uH dM \leq \frac{1}{\sqrt{6}}\int_\Omega \rho H^{1/2}\lan P_1(\n u),\n u\ran^{1/2} dM.
\]
Therefore, by using (\ref{diam}),
\[
\int_\Omega uH dM \leq \frac{r_\Omega}{\sqrt{6}}\int_\Omega H^{1/2}\lan P_1(\n u),\n u\ran^{1/2} dM.
\]
This proves the Proposition \ref{mod.poincare}.
\end{proof}


We now prove Theorem \ref{stability}.

\begin{proof}[Proof of Theorem \ref{stability}.]
Without loss of generality, choose an orientation of $M^3$ such that the  mean curvature $H>0.$ The hypothesis $R=0$ and the inequality (\ref{desig}), p. \pageref{desig}, imply $K\leq0.$ Since $L_1(\frac{1}{2}g^2)=gL_1g + \lan P_1(\n g),\n g\ran$ for any $g$ with compact support, after integrating and using the divergence theorem, stability becomes equivalent to
\begin{equation}\label{ineq.stab}
-3\int_{\Omega} K g^2 dM \leq \int_{\Omega}\lan P_1(\n g),\n g\ran dM,
\end{equation}
for any smooth function $g:\Omega\subset M\ria\R$ with compact support, satisfying $g|_{\partial\Omega}=0,$ where $\Omega$ is a regular domain. The proof will be made by contradiction. Suppose $\Omega$ is unstable. Then there exists a smooth function $g:\Omega\ria\R,$ with compact support, satisfying $g|_{\partial\Omega}=0,$ such that
\begin{equation}\label{eq.est.1}
-3\int_{\Omega}Kg^2 dM > \int_{\Omega} \lan P_1 (\n g), \n g\ran dM.
\end{equation}
Choosing $u=g^2$ in the inequality (\ref{new.sobolev.2}) of Proposition \ref{mod.poincare}, we have 
\[
\int_{\Omega} g^2 HdM\leq \frac{2r_\Omega}{\sqrt{6}}\int_{\Omega} H^{1/2}g\lan P_1(\n g),\n g\ran^{1/2}dM.
\]
By using the Cauchy-Schwarz inequality in the right hand side of the inequality above, we have
\[
\int_{\Omega} H^{1/2}g\lan P_1(\n g),\n g\ran^{1/2}dM\leq \left(\int_{\Omega} g^2 HdM\right)^{1/2}\left(\int_{\Omega} \lan P_1(\n g),\n g\ran dM\right)^{1/2},
\]
and therefore,
\[
\left(\int_{\Omega} g^2 HdM\right)^{1/2}\leq\frac{2r_\Omega}{\sqrt{6}}\left(\int_{\Omega} \lan P_1(\n g),\n g\ran dM\right)^{1/2}.
\]
By using the hypothesis (\ref{eq.est.1}), we have
\[
\begin{split}
\int_{\Omega} g^2 H dM &\leq \frac{2r_\Omega^2}{3} \int_{\Omega} \lan P_1(\n g),\n g\ran dM\\
&<\frac{2r_\Omega^2}{3}\int_\Omega (-3K)g^2dM\\
&\leq \frac{2r_\Omega^2}{3}\sup_\Omega\left(\dfrac{-3K}{H}\right)\int_\Omega g^2 H dM,\\  
\end{split}
\]
i.e.,
\[
1<\frac{2r_\Omega^2}{3}\sup_\Omega\left(\dfrac{-3K}{H}\right)
\]
which is a contradiction.
\end{proof}



\bibliographystyle{amsplain}


\end{document}